\documentclass[12pt, reqno, a4paper]{amsart}

\usepackage{amssymb}
\usepackage{hyperref}

\usepackage[english]{babel}

\theoremstyle{plain}
\newtheorem{proposition}{Proposition}
\newtheorem*{corollary}{Corollary}
\newtheorem{lemma}{Lemma}
\newtheorem{theorem}{Theorem}

\theoremstyle{remark}
\newtheorem*{remark}{Remark}

\theoremstyle{definition}

\newtheorem{example}{Example}

\DeclareMathOperator{\Id}{Id}

\DeclareMathOperator{\Tor}{Tor}

\DeclareMathOperator{\ch}{char}

\DeclareMathOperator{\UT}{UT}

\DeclareMathOperator{\End}{End}

\DeclareMathOperator{\sign}{sign}

\DeclareMathOperator{\Hom}{Hom}

\begin{document}

\title{$\mathbb ZS_n$-modules and polynomial identities with integer coefficients}
\author{Alexey Gordienko}
\author{Geoffrey Janssens}
\address{Vrije Universiteit Brussel, Belgium}
\email{alexey.gordienko@vub.ac.be, Geoffrey.Janssens@vub.ac.be} 


\begin{abstract} 
We show that, like in the case of algebras over fields, the study of multilinear polynomial identities of unitary rings can be reduced to the study of proper polynomial identities. In particular, the factors of series of $\mathbb ZS_n$-submodules in 
the $\mathbb ZS_n$-modules of multilinear polynomial functions can be derived by the analog of Young's (or Pieri's) rule from
the factors of series in the corresponding  $\mathbb ZS_n$-modules of proper polynomial functions.

As an application, we calculate the codimensions and a basis of multilinear polynomial
identities of unitary rings of upper triangular $2\times 2$ matrices and infinitely generated Grassmann algebras over unitary rings. In addition, we calculate the factors of series of $\mathbb ZS_n$-submodules for these algebras. 

Also we establish relations between codimensions of rings and codimensions of algebras and show that the analog of Amitsur's conjecture holds in all torsion-free rings,
and all torsion-free rings with $1$ satisfy the analog of Regev's conjecture.
\end{abstract}


\thanks{
The first author was supported by Fonds Wetenschappelijk Onderzoek~--- Vlaanderen Pegasus Marie Curie post doctoral fellowship (Belgium) and RFBR grant 13-01-00234a (Russia).
}

\maketitle

Polynomial identities and their numeric and representational characteristics are well studied in the case of algebras over fields of characteristic zero (see e.g.~\cite{DrenKurs, ZaiGia}). 
However, polynomial identities in rings also play an important role~\cite{RowenPIR}.
The systematic study of multilinear polynomial identities started in 1950 by A.\,I.~Mal'cev~\cite{MalcevIden}
and W.~Specht~\cite{Specht}. In his paper, W.~Specht considered polynomial identities with integer coefficients.
This article is devoted to numeric and representational characteristics of polynomial identities with integer coefficients.

In Propositions~\ref{PropositionIntAbsence}--\ref{PropositionIntPIoverZp} we prove basic facts about codimensions of polynomial identities with integer coefficients in the case when the ring is an algebra over a field.
In Theorem~\ref{TheoremFCodimofaRing} we consider algebras obtained from rings by the tensor product of the ring and a field. As a consequence of Theorem~\ref{TheoremFCodimofaRing}, we derive
the analog of Amitsur's conjecture for all torsion-free rings
and the analog of Regev's conjecture for all torsion-free rings with $1$.
In Theorem~\ref{TheoremCodimProperAndOrdinary}
we show that the ordinary codimensions of rings with $1$ can be calculated using
their proper codimensions by the same formula as in the case of algebras over
fields~\cite[Theorem~4.3.12 (ii)]{DrenKurs}. 
In Theorem~\ref{TheoremZSnOrdinaryProper} we prove an analog of Drensky's theorem~\cite[Theorem 12.5.4]{DrenKurs} that establishes a relation between $\mathbb ZS_n$-modules corresponding to proper and ordinary polynomial identities of unitary rings. In order to apply Theorem~\ref{TheoremZSnOrdinaryProper}, we show
 that the analog of Young's (or Pieri's) rule holds for $\mathbb ZS_n$-modules too (Theorem~\ref{TheoremYoungsRule}).  

In the case of algebras over fields, few examples are known where the
numeric and representational characteristics can be precisely evaluated.
Among those are the algebras of upper triangular matrices~\cite{Kalyulaid, MalcevYuN, Polin, Siderov} and the Grassmann algebra~\cite{KrakReg}. In Sections~\ref{SectionUT2R} and~\ref{SectionGrassmannR} we apply the results obtained in the preceding sections, to evaluate the basis of multilinear polynomial identities with integer coefficients and calculate their numeric and representational characteristics for unitary rings of upper triangular $2\times 2$ matrices and infinitely generated Grassmann algebras over unitary rings.

\section{Introduction}

Let $R$ be a ring. If $R$ has a unit element $1_R$, then the number $$\ch R := \min\lbrace
 n\in\mathbb N \mid n1_R = 0\rbrace = \min\lbrace
n\in\mathbb N \mid na = 0 \text{ for all } a\in R\rbrace$$ is called the \textit{characteristic} of $R$. 
(As usual, if $n1_R \ne 0$ for all $n\in\mathbb N$, then $\ch R := 0$.)

Let $\mathbb Z\langle X \rangle$ be the free associative ring without $1$ on the countable set 
$X=\lbrace x_1, x_2, \dots\rbrace$, i.e., the ring of polynomials in non-commuting variables
from $X$ without a constant term.

Let $I$ be an ideal of a ring $R$. We say that $I$ is a \textit{T-ideal} of $R$ if $\varphi(I)\subseteq I$
for all $\varphi \in \End(R)$.  We say that $f \in \mathbb Z\langle X \rangle$
is a \textit{polynomial identity} of $R$ \textit{with integer coefficients} if $f(a_1, \dots, a_n)=0$ for all $a_i \in R$.
 In other words, $f$ is a polynomial identity if $\psi(f)=0$ for all $\psi \in\Hom(\mathbb Z\langle X \rangle, R)$. Note that the set $\Id(R, \mathbb Z)$ of polynomial identities of $R$ with integer coefficients is a $T$-ideal of $\mathbb Z\langle X \rangle$.

Let $P_n(\mathbb Z)$ be the additive subgroup of $\mathbb Z\langle X \rangle$ generated by $x_{\sigma(1)} x_{\sigma(2)} \dots x_{\sigma(n)}$, $\sigma \in S_n$. (Here $S_n$ is the $n$th symmetric group, $n\in\mathbb N$.)
Then $\frac{P_n(\mathbb Z)}{P_n(\mathbb Z) \cap \Id(R, \mathbb Z)}$ is a finitely generated Abelian group which is the direct sum of free 
and primary cyclic groups: $$\frac{P_n(\mathbb Z)}{P_n(\mathbb Z) \cap \Id(R, \mathbb Z)} \cong 
\underbrace{\mathbb Z \oplus \dots \oplus \mathbb Z}_{c_n(R, 0)} \oplus \bigoplus_{\substack{p\text{ is a prime}\\ \text{number}}}
\ \bigoplus_{k\in\mathbb N} \Bigl(\underbrace{\mathbb Z_{p^k}\oplus \dots \oplus \mathbb Z_{p^k}}_{c_n(R, p^k)}\Bigr).$$

We call the numbers $c_n(R, q)$ the \textit{codimensions} of polynomial identities of $R$
with integer coefficients.

Note that the symmetric group $S_n$ is acting on $\frac{P_n(\mathbb Z)}{P_n(\mathbb Z) \cap \Id(R, \mathbb Z)}$ by permutations
of variables, i.e., $\frac{P_n(\mathbb Z)}{P_n(\mathbb Z) \cap \Id(R, \mathbb Z)}$ is a $\mathbb ZS_n$-module.  We refer to $\frac{P_n(\mathbb Z)}{P_n(\mathbb Z) \cap \Id(R, \mathbb Z)}$ 
as the {\itshape $\mathbb ZS_n$-module of ordinary multilinear polynomial functions} on $R$.

Denote by $\Gamma_n(\mathbb Z)$ the subgroup of $P_n(\mathbb Z)$ that consists of {\itshape proper polynomials}, i.e linear combinations of products of long commutators. (All long commutators in the article are left normed, e.g. $[x,y,z,t]:=[[[x,y],z],t]$.)
Then $\Gamma_n(\mathbb Z)$ is a $\mathbb ZS_n$-submodule of $P_n(\mathbb Z)$. Obviously, $\Gamma_1(\mathbb Z) = 0$.

Analogously, we define the \textit{codimensions} $\gamma_n(R, q)$ of proper polynomial identities of $R$:
$$\frac{\Gamma_n(\mathbb Z)}{\Gamma_n(\mathbb Z) \cap \Id(R, \mathbb Z)} \cong 
\underbrace{\mathbb Z \oplus \dots \oplus \mathbb Z}_{\gamma_n(R, 0)} \oplus \bigoplus_{\substack{p\text{ is a prime}\\ \text{number}}}
\ \bigoplus_{k\in\mathbb N} \Bigl(\underbrace{\mathbb Z_{p^k}\oplus \dots \oplus \mathbb Z_{p^k}}_{\gamma_n(R, p^k)}\Bigr).$$
If $R$ has a unit element $1_R$, then, by the definition, $\gamma_0(R, q)$ is the number of $\mathbb Z_q$ in the decomposition of the cyclic additive subgroup of $R$ generated by $1_R$.
 We refer to $\frac{\Gamma_n(\mathbb Z)}{\Gamma_n(\mathbb Z) \cap \Id(R, \mathbb Z)}$ 
as the {\itshape $\mathbb ZS_n$-module of proper multilinear polynomial functions} on $R$.

If $A$ is an algebra over a field $F$, then we can consider codimensions $c_n(A, F) := \dim \frac{P_n(F)}{P_n(F) \cap \Id(A, F)}$ of polynomial
identities of $A$ with coefficients from $F$. (See~\cite[Definition 4.1.1]{ZaiGia}.) 
Here $\Id(A, F) \subset F\langle X\rangle$
 is the set of polynomial identities of $A$ with coefficients from $F$, and $P_n(F)$ is the subspace of
$F\langle X\rangle$  generated by $x_{\sigma(1)} x_{\sigma(2)} \dots x_{\sigma(n)}$, $\sigma \in S_n$. 
The subspace of $P_n(F)$ consisting of proper polynomials, is denoted by $\Gamma_n(F)$.

We say that $\lambda=(\lambda_1, \dots, \lambda_s)$
is a \textit{(proper or ordered) partition} of $n$ and write $\lambda\vdash n$ if $\lambda_1 \geqslant \lambda_2 \geqslant \dots \geqslant \lambda_s > 0$,
$\lambda_i\in\mathbb N$, and $\sum_{i=1}^s \lambda_i = n$.
In this case we write $\lambda \vdash n$. For our convenience, we assume $\lambda_i=0$
for all $i > s$.

We say that $\mu=(\mu_1, \dots, \mu_s)$
is an \textit{unordered partition} of $n$ if $\mu_i\in\mathbb N$ and $\sum_{i=1}^s \mu_i = n$.
In this case we write $\mu \vDash n$. Again, for our convenience, we assume $\mu_i=0$
for all $i > s$.

For every ordered or unordered partition $\lambda$ one can assign the \textit{Young diagram} $D_\lambda$
which contains $\lambda_k$ boxes in the $k$th row. If $\lambda$ is unordered, then $D_\lambda$
is called \textit{generalized}. A Young diagram filled with numbers is called a \textit{Young tableau}.
A tableau corresponding to $\lambda$ is denoted by $T_\lambda$.

In the representation theory of symmetric groups, partitions and their
  Young diagrams are widely used. (See~\cite{DrenKurs, ZaiGia} for applications to PI-algebras.)
   Let  $a_{T_{\lambda}} = \sum_{\pi \in R_{T_\lambda}} \pi$
and $b_{T_{\lambda}} = \sum_{\sigma \in C_{T_\lambda}}
 (\sign \sigma) \sigma$
be symmetrizers corresponding to a Young tableau~$T_\lambda$, $\lambda \vdash n$.
Then $S(\lambda) := (\mathbb Z S_n) b_{T_\lambda} a_{T_\lambda}$
 is the corresponding Specht module. Moreover, modules $S(\lambda)$
 that correspond to different $T_\lambda$ but the same $\lambda$, are isomorphic too.
  (The proof is analogous to the case of fields.)
 Though $S(\lambda)$ are not irreducible over $\mathbb Z$ and even contain no irreducible $\mathbb Z S_n$-submodules (it is sufficient to consider the submodule $0 \ne 2M \subsetneqq M$ for any submodule $M \subseteq S(\lambda)$), we will use
 them in order to describe the structure of $\frac{P_n(\mathbb Z)}{P_n(\mathbb Z) \cap \Id(R, \mathbb Z)}$.

\section{Codimensions of algebras over fields}

Every algebra over a field can be treated as a ring. Therefore, we have to deal with two different types of codimensions. Here we establish a relation between them.

\begin{proposition}\label{PropositionIntAbsence} Let $A$ be an algebra over a field $F$.
Then $c_n(A, q)=0$ for all $n\in\mathbb N$ and $q \ne \ch F$.
\end{proposition}
\begin{proof}
Note that $(\ch F) f \in \Id(R, \mathbb Z)$ for all $f\in \mathbb Z\langle X \rangle$.
Hence $\ch F > 0$ implies $c_n(A, q)=0$ for all $n\in\mathbb N$ and $q \ne \ch F$.
If $\ch F = 0$, then every $q=p^k \ne 0$ is invertible and $c_n(A, q)=0$ for all $n\in\mathbb N$ and $q \ne 0$ too.
\end{proof}
\begin{proposition}\label{PropositionIntPIoverQ}
Let $A$ be an algebra over 
a field $F$, $\ch F = 0$.
Then $c_n(A, F)
\leqslant c_n(A, 0)$ for all $n\in\mathbb N$.
Moreover, $c_n(A, \mathbb Q)
= c_n(A, 0)$ for all $n\in\mathbb N$.
\end{proposition}
\begin{proof}
By Proposition~\ref{PropositionIntAbsence}, $\frac{P_n(\mathbb Z)}{P_n(\mathbb Z) \cap \Id(A, \mathbb Z)}$ is
a free Abelian group. Let $f_1, \dots, f_s$ be the preimages of its free generators in $P_n(\mathbb Z)$.
Note that $P_n(\mathbb Z) \subset P_n(\mathbb Q) \subseteq P_n(F)$ and for every $\sigma \in S_n$ the monomial $x_{\sigma(1)}x_{\sigma(2)}\dots x_{\sigma(n)}$ can be expressed as a linear combination with integer coefficients of
$f_1, \dots, f_s$ and an element of $P_n(\mathbb Z) \cap \Id(A, \mathbb Z)$.
Hence the images of $f_1, \dots, f_s$ generate $\frac{P_n(F)}{P_n(F) \cap \Id(A, F)}$
and $c_n(A, F) \leqslant c_n(A, 0)=s$. 

Suppose $f_1, \dots, f_s$ are linearly dependent modulo $\Id(A, \mathbb Q)$. In this case $\frac{r_1}{q_1}f_1+\dots +\frac{r_1}{q_1}f_s \in \Id(A, \mathbb Q)$ for some $q_i \in\mathbb N$, $r_i \in \mathbb Z$.
Thus $$f:=r_1\left(\prod_{i=2}^s q_i\right)f_1 + r_1q_1\left(\prod_{i=3}^s q_i\right)f_2+\dots
+ r_s \left(\prod_{i=1}^{s-1} q_i\right)f_s \in \Id(A, \mathbb Q).$$
However, $f\in \mathbb Z\langle X\rangle$. Hence $f \in \Id(A, \mathbb Z)$
and all $r_i = 0$ since $f_i$ are linearly independent modulo $\Id(A, \mathbb Z)$.
Therefore, the images of $f_1, \dots, f_s$ form a basis of $\frac{P_n(\mathbb Q)}{P_n(\mathbb Q) \cap \Id(A, \mathbb Q)}$ and $c_n(A, \mathbb Q)=c_n(A, 0)=s$.
\end{proof}

The next example shows that in the case $F\supsetneqq \mathbb Q$  we
could have $c_n(A, F) < c_n(A, \mathbb Q) = c_n(A, 0)$.

\begin{example}
Note that $P_3(\mathbb Q)\cong \mathbb QS_3 \cong S^{\mathbb Q}(3) \oplus S^{\mathbb Q}(2,1) \oplus S^{\mathbb Q}(2,1)\oplus S^{\mathbb Q}(1^3)$.
Let $a\in \mathbb QS_3$ such that $S^{\mathbb Q}(2,1)=\mathbb QS_3 a$. Denote by $f_1$ and $f_2$ the polynomials
that correspond to $a$ in the copies of $S^{\mathbb Q}(2,1)$ in $P_3(\mathbb Q)$.
Let $F=\mathbb Q(\sqrt 2)$.
Consider the $T$-ideal $I$ of $ F\langle X \rangle$ generated by $(f_1 + {\sqrt 2} f_2)$.
We claim that $c_3(F\langle X \rangle/I, F) = 4 < c_3(F\langle X \rangle/I, \mathbb Q)=6$.
\end{example}
\begin{proof}
First we notice that $P_3(F)\cap \Id(F\langle X \rangle/I, F)
= F S_3\cdot (f_1 + {\sqrt 2} f_2) \cong S^F(2,1)$.
Hence by the hook formula, $c_3(F\langle X \rangle/I, F) = 6-2 = 4$.
However, $P_3(\mathbb Q)\cap \Id(F\langle X \rangle/I, \mathbb Q)
= P_3(\mathbb Q)\cap F S_3 (f_1 + {\sqrt 2} f_2)=0$.
Indeed, suppose $f = b (f_1  + {\sqrt 2} f_2) \in P_3(\mathbb Q)$
for some $b\in F S_3$. Note that $b=b_1 +\sqrt 2 b_2$ where $b_1, b_2\in \mathbb QS_3$.
Therefore, $f = (b_1 +\sqrt 2 b_2) (f_1  + {\sqrt 2} f_2) = (b_1f_1+2 b_2 f_2) + \sqrt 2 (b_1 f_2 + b_2 f_1)$ and $f \in P_3(\mathbb Q)$ implies $b_1 f_2 + b_2 f_1 = 0$.
Recall that $\mathbb Q S_3 f_1 \oplus \mathbb Q S_3 f_2$ is the direct sum of $\mathbb Q S_3$-submodules.
Hence $b_1f_2 = b_2 f_1 = 0$. However, $\mathbb Q S_3 f_1 \cong \mathbb Q S_3 f_2$.
Thus $b_1f_1 = b_2 f_2 = 0$ too, $f=0$,
 $P_3(\mathbb Q)\cap \Id(F\langle X \rangle/I, \mathbb Q) = 0$
and $c_3(F\langle X \rangle/I, \mathbb Q)=6$.
\end{proof}

The result, analogous to Proposition~\ref{PropositionIntPIoverQ}, holds in a positive characteristic.

\begin{proposition}\label{PropositionIntPIoverZp}
Let $A$ be an algebra over 
a field $F$, $\ch F = p$.
Then $c_n(A, F)
\leqslant c_n(A, p)$ for all $n\in\mathbb N$.
Moreover, $c_n(A, \mathbb Z_p)
= c_n(A, p)$ for all $n\in\mathbb N$.
\end{proposition}
\begin{proof}
By Proposition~\ref{PropositionIntAbsence}, $\frac{P_n(\mathbb Z)}{P_n(\mathbb Z) \cap \Id(A, \mathbb Z)}$ is
the direct sum of copies of $\mathbb Z_p$. Let $f_1, \dots, f_s$ be the preimages of their standard generators in $P_n(\mathbb Z)$.
Note that $P_n(\mathbb Z_p)$ is an image of $P_n(\mathbb Z)$
under the natural homomorphism, $P_n(\mathbb Z_p) \subseteq P_n(F)$ and for every $\sigma \in S_n$ the monomial $x_{\sigma(1)}x_{\sigma(2)}\dots x_{\sigma(n)}$ can be expressed as a linear combination with integer coefficients of
$f_1, \dots, f_s$ and an element of $P_n(\mathbb Z) \cap \Id(A, \mathbb Z)$.
Hence the images of $f_1, \dots, f_s$ generate $\frac{P_n(F)}{P_n(F) \cap \Id(A, F)}$
and $c_n(A, F) \leqslant c_n(A, p)$. 

Suppose $f_1, \dots, f_s$ are linearly dependent modulo $\Id(A, \mathbb Z_p)$. In this case $ \bar m_1 f_1+\dots +\bar m_s f_s \in \Id(A, \mathbb Z_p)$ for some $m_i \in \mathbb Z$.
Thus $m_1 f_1+\dots +m_s f_s \in \Id(A, \mathbb Z)$
and all $m_i \in p\mathbb Z$ since $f_i$ generate modulo $\Id(A, \mathbb Z)$
the direct sum of copies of $\mathbb Z_p$.
Therefore, the images of $f_1, \dots, f_s$ form a basis of $\frac{P_n(\mathbb Z_p)}{P_n(\mathbb Z_p) \cap \Id(A, \mathbb Z_p)}$ and $c_n(A, \mathbb Z_p)=c_n(A, p)=s$.
\end{proof}

The next result is concerned with the extension of a ring to an algebra over a field.

\begin{theorem}\label{TheoremFCodimofaRing}
Let $R$ be a ring and let $F$ be a field.
Then $$c_n(R \mathbin{\otimes_{\mathbb Z}} F, F)
= \left\lbrace
\begin{array}{ccc} c_n(R/{\Tor R}, 0) & if & \ch F = 0, \\ 
c_n(R/p R, p)& if & \ch F = p \end{array}\right.$$
where $\Tor R := \lbrace r \in R \mid mr=0 \text{ for some } m\in\mathbb N\rbrace$ is the torsion of $R$.
\end{theorem}

First, we prove the following lemma

\begin{lemma}\label{LemmaTensorWithAField}
Let $R$ be a ring and let $F$ be a field.
Then $$R \otimes 1_F \cong \left\lbrace
\begin{array}{ccc} R/{\Tor R} & if & F = \mathbb Q, \\ 
R/p R & if & F = \mathbb Z_p \end{array}\right.$$
where $R \otimes 1_F \subseteq R \mathbin{\otimes_{\mathbb Z}} F$ is a subring.
\end{lemma}
\begin{proof}
Consider the natural homomorphism $\varphi \colon R \to R \otimes 1_F$
where $\varphi(a)=a\otimes 1_F$, $a\in R$.

Suppose $F = \mathbb Q$. 
If $ma = 0$ for some $m\in\mathbb N$ and $a\in R$,
then $\varphi(a)=a \otimes 1_{\mathbb Q} = ma \otimes \frac{1_{\mathbb Q}}{m} = 0$.
Hence $\Tor R \subseteq \ker \varphi$. 
We claim that $\ker \varphi = \Tor R$. 

Let $a \in \ker \varphi$, i.e., $a \otimes   1_{\mathbb Q} = 0$. By one of the definitions of the tensor product, \begin{equation*}\begin{split}(a, 1_{\mathbb Q})=\sum_i \ell_i((a_i+b_i, q_i)-(a_i,q_i)-(b_i,q_i))+\\ \sum_i m_i((c_i, s_i+t_i)-(c_i,s_i)-(c_i,t_i)) +
\sum_i n_i((k_i d_i, u_i) - (d_i, k_i u_i))\end{split}\end{equation*} holds for some $a_i,b_i,c_i,d_i\in R$, $k_i, \ell_i,m_i,n_i\in\mathbb Z$,
and $q_i, s_i, t_i, u_i \in \mathbb Q$ in the free $\mathbb Z$-module $H_{R\times \mathbb Q}$
with the basis $R\times \mathbb Q$.
We can find such $m \in\mathbb N$ that all $mq_i,ms_i,mt_i, mu_i \in \mathbb Z$.
Then \begin{equation*}\begin{split}(a, m)=\sum_i \ell_i((a_i+b_i, mq_i)-(a_i,mq_i)-(b_i,mq_i))+\\ \sum_i m_i((c_i, ms_i+mt_i)-(c_i,ms_i)-(c_i,mt_i)) +
\sum_i n_i((k_i d_i, mu_i) - (d_i, k_i mu_i))\end{split}\end{equation*} holds in the free $\mathbb Z$-module $H_{R\times \mathbb Z}$
with the basis $R\times \mathbb Z$. Note that in the right hand side of the latter equality
we have a relation in $R \mathbin{\otimes_{\mathbb Z}} \mathbb Z$.
Hence $a\otimes m = 0$ in $R \mathbin{\otimes_{\mathbb Z}} \mathbb Z \cong R$
and $ma=0$. Thus $a\in \Tor R$. Therefore, $\ker \varphi = \Tor R$ and $R \otimes 1_{\mathbb Q} \cong R/{\Tor R}$.

Suppose $F = \mathbb Z_p$. Then $\varphi(pR)= R\otimes p1_{\mathbb Z_p} = 0$
and $pR \subseteq \ker \varphi$.
Let $a\in\ker\varphi$, i.e., $a \otimes 1_{\mathbb Z_p} = 0$.
Then \begin{equation*}\begin{split}(a, 1_{\mathbb Z_p})=\sum_i q_i((a_i+b_i, \bar\ell_i)-(a_i,\bar\ell_i)-(b_i,\bar\ell_i))+\\ \sum_i s_i((c_i, \bar m_i+ \bar n_i)-(c_i, \bar m_i)-(c_i, \bar n_i)) + \sum_i t_i((k_i d_i, \bar u_i) - (d_i, k_i \bar u_i)) \end{split}\end{equation*} holds for some $a_i,b_i,c_i, d_i\in R$
and $k_i, \ell_i,m_i,n_i, q_i, s_i, t_i, u_i \in \mathbb Z$ in the free $\mathbb Z$-module $H_{R\times \mathbb Z_p}$
with the basis $R\times \mathbb Z_p$.
Note that $H_{R\times \mathbb Z_p}$ is the factor module of $H_{R\times \mathbb Z}$
by the subgroup $\langle (a, m)-(a, m+p) \mid  a \in R,\ m\in\mathbb Z\rangle_{\mathbb Z}$.
Hence \begin{equation*}\begin{split}(a, 1_\mathbb Z)=\sum_i q_i((a_i+b_i, \ell_i)-(a_i,\ell_i)-(b_i,\ell_i))+ 
\sum_i s_i((c_i, m_i+n_i)-(c_i,m_i)-(c_i,n_i))+\\ \sum_i t_i((k_i d_i,  u_i) - (d_i, k_i  u_i))+ \sum_i \alpha_i((r_i, \beta_i)-(r_i, \beta_i+p))\end{split}\end{equation*}
holds in $H_{R\times \mathbb Z}$ for some $r_i \in R$ and $\alpha_i, \beta_i\in\mathbb Z$.
Thus $a\otimes 1_\mathbb Z = \sum_i \alpha_i r_i \otimes p$.
Now we use the isomorphism $R \mathbin{\otimes_{\mathbb Z}} \mathbb Z \cong R$
 and get  $a = \sum_i \alpha_i r_i p \in pR$. Therefore, $\ker \varphi = pR$
 and
$R \otimes 1_{\mathbb Z_p} \cong R/{p R}$.
\end{proof}
\begin{proof}[Proof of Theorem~\ref{TheoremFCodimofaRing}.]
Recall that $R \mathbin \otimes 1_F$ is a subring of $R \mathbin{\otimes_{\mathbb Z}} F$.
Hence $P_n(\mathbb Z) \cap \Id(R \mathbin{\otimes_{\mathbb Z}} F, \mathbb Z)
  \subseteq P_n(\mathbb Z) \cap \Id(R \otimes 1_F, \mathbb Z)$.
Conversely, $P_n(\mathbb Z) \cap \Id(R \mathbin{\otimes_{\mathbb Z}} F, \mathbb Z)
  \supseteq P_n(\mathbb Z) \cap \Id(R \otimes 1_F, \mathbb Z)$ since $R \mathbin\otimes 1_F$ generates $R \mathbin{\otimes_{\mathbb Z}} F$
as an $F$-vector space. Therefore, $c_n(R \otimes 1_F, \ch F) = c_n(R \mathbin{\otimes_{\mathbb Z}} F, \ch F)$ and we get Theorem~\ref{TheoremFCodimofaRing} for $F=\mathbb Q$ and $F=\mathbb Z_p$ from Lemma~\ref{LemmaTensorWithAField} and Propositions~\ref{PropositionIntPIoverQ}, \ref{PropositionIntPIoverZp}.
The general case follows from the fact that $(R \mathbin{\otimes_{\mathbb Z}} F) \mathbin{\otimes_{F}} K \cong R \mathbin{\otimes_{\mathbb Z}} K$ (as a $K$-algebra) for any field
extension $K \supseteq F$ and, by~\cite[Theorem~4.1.9]{ZaiGia}, $$c_n(R \mathbin{\otimes_{\mathbb Z}} K, K)=c_n((R \mathbin{\otimes_{\mathbb Z}} F) \mathbin{\otimes_F} K, K)= c_n(R \mathbin{\otimes_{\mathbb Z}} F, F).$$
\end{proof}

\begin{corollary}
Let $R$ be a torsion-free ring satisfying a non-trivial polynomial identity. 
Then 
\begin{enumerate} \item either $c_n(R, 0) = 0$ for all $n\geqslant n_0$, $n_0\in\mathbb N$,
or there exist $d\in\mathbb N$, $C_1, C_2 > 0$, $q_1, q_2 \in\mathbb R$
such that $C_1 n^{q_1} d^n \leqslant c_n(R, 0) \leqslant C_2 n^{q_2} d^n$
for all $n\in\mathbb N$; 
in particular, polynomial identities of $R$ satisfy the analog of \textit{Amitsur's
conjecture}, i.e., there exists $\lim_{n\to\infty} \sqrt[n]{c_n(R, 0)} \in\mathbb Z_+$;
\item
if $R$ contains $1$, then there exist $C>0$ and $q\in\mathbb Z$ such that
$c_n(R, 0) \sim C n^{\frac{q}{2}} d^n$ as $n\to\infty$, i.e., the analog of \textit{Regev's
conjecture} holds in $R$. (We write $f \sim g$ if $\lim \frac f g = 1$.)
\end{enumerate}
\end{corollary}
\begin{proof}
By Theorem~\ref{TheoremFCodimofaRing}, $c_n(R, 0) = c_n(R \mathbin{\otimes_{\mathbb Z}} \mathbb Q, \mathbb Q)$. Now we apply \cite[Theorem~6.5.2]{ZaiGia} and~\cite[Theorem~4.2.2]{BereleInfDim}.
\end{proof}
\begin{remark}
If $R$ is a torsion-free ring, then $c_n(R, q)=0$ for all $q\ne 0$
since $f \in \Id(R, \mathbb Z)$
for all $f\in\mathbb Z\langle X \rangle$
such that $mf \in \Id(R, \mathbb Z)$ for some $m\in\mathbb N$.
\end{remark}

We conclude the section with an example.

\begin{example}
Let $R = \bigoplus_{k=1}^\infty \mathbb Z_{2^k}$.
Then $c_n(R, 0) = 1$ and $c_n(R, q) = 0$
for all $q\ne 0$ and $n\in\mathbb N$.
Although $mR \ne 0$ for all $m\in \mathbb N$, $R \mathbin{\otimes_{\mathbb Z}} \mathbb Q = 0$ and $c_n(R \mathbin{\otimes_{\mathbb Z}} \mathbb Q, \mathbb Q)=0$
for all $n\in\mathbb N$.
\end{example}
\begin{proof}
The ring $R$ is commutative. Hence all monomials from $P_n(\mathbb Z)$
are proportional to $x_1 x_2 \dots x_n$ modulo $\Id(R, \mathbb Z)$.
However, $m x_1 x_2 \dots x_n \notin \Id(R, \mathbb Z)$
for all $m\in\mathbb N$. (It is sufficient to substitute
$x_1 = x_2 = \dots = x_n = \bar 1_{\mathbb Z_{2^k}}$ for
$2^k > m$.) Thus $\frac{P_n(\mathbb Z)}{P_n(\mathbb Z)\cap \Id(R, \mathbb Z)}
\cong \mathbb Z$ and $c_n(R, 0) = 1$ and $c_n(R, q) = 0$
for all $q\ne 0$ and $n\in\mathbb N$.
However $a \otimes q = 2^k a \otimes \frac{q}{2^k}$
for all $a\in R$, $q\in\mathbb Q$, and $k\in \mathbb N$.
Choosing $k$ sufficiently large, we get $a \otimes q = 2^k a \otimes \frac{q}{2^k} = 0$. Thus $R \mathbin{\otimes_{\mathbb Z}} \mathbb Q = 0$ and $c_n(R \mathbin{\otimes_{\mathbb Z}} \mathbb Q, \mathbb Q)=0$
for all $n\in\mathbb N$.
\end{proof}

\section{Relation between $\mathbb ZS_n$-modules of proper and ordinary polynomial functions}

First, we describe the relation between proper and ordinary codimensions.

\begin{theorem}\label{TheoremCodimProperAndOrdinary}
Let $R$ be a unitary ring. Then
$c_n(R, q) = \sum_{j=0}^n \tbinom{n}{j}\gamma_j(R, q)$
for every $n\in \mathbb N$ and $q \in \lbrace p^k \mid p,k\in\mathbb N,\ p \text{ is prime } \rbrace \cup \lbrace 0\rbrace$.
\end{theorem}
\begin{proof}
First, we notice that \begin{equation}\label{EqPnDecomp}P_n(\mathbb Z) = \bigoplus_{k=0}^n \bigoplus_{1\leqslant i_1 <
i_2 < \dots < i_k \leqslant n} x_{i_1} x_{i_2}\dots x_{i_k} \, \sigma_{i_1, \dots, i_k} 
\Gamma_{n-k}(\mathbb Z) \text{ (direct sum of $\mathbb Z$-modules) }\end{equation}
where $\Gamma_0(\mathbb Z) := \mathbb Z$ and $\sigma_{i_1, \dots, i_k} \in S_n$ is any permutation such that $\sigma((n-k)+j)=i_j$ for all $1\leqslant j\leqslant k$.

One way to prove~(\ref{EqPnDecomp}) is to use the Poincar\'e~--- Birkhoff~--- Witt theorem
for Lie algebras over rings~\cite[Theorem 2.5.3]{Bahturin}.

Another way is to show this explicitly in the spirit of Specht~\cite{Specht}. Using the equalities $yx = [y,x]+xy$ and $[\dots,\dots]x=x[\dots,\dots]+[[\dots, \dots],x]$,
we can present every polynomial from $P_n(\mathbb Z)$ as a linear combination of
polynomials $x_{i_1}x_{i_2}\dots x_{i_k}\, f$ where $1\leqslant i_1 <
i_2 < \dots < i_k \leqslant n$ and $f$ is a proper multilinear
polynomial of degree $(n-k)$ in the variables from the set $\lbrace x_1, x_2, \dots, x_n \rbrace \backslash 
 \lbrace x_{i_1}, x_{i_2}, \dots, x_{i_k} \rbrace$.
 In other words, $f \in \sigma_{i_1, \dots, i_k} \Gamma_{n-k}(\mathbb Z)$.
In order to check that the sum in~(\ref{EqPnDecomp}) is direct, we consider
a linear combination of $x_{i_1} x_{i_2}\dots x_{i_k}
\sigma_{i_1, \dots, i_k} f $ where $f\in \Gamma_{n-k}(\mathbb Z)$,
for different $k$ and $i_j$ and choose the term $g := x_{i_1} x_{i_2}\dots x_{i_k} \sigma_{i_1, \dots, i_k} f $ with the greatest $k$
among the terms with a nonzero coefficient.
Then we substitute $x_{i_1}=x_{i_2}=
\dots = x_{i_k}=1$ and $x_j = x_j$ for the rest of the variables.
(We assume that we are working in the free ring with $1$
on the set $X=\lbrace x_1, x_2, \dots \rbrace$.)
All the other terms vanish and we get $f=0$.
Therefore, the sum is direct and~(\ref{EqPnDecomp}) holds.

Substituting $x_{i_1}=x_{i_2}=
\dots = x_{i_k}=1_R$
and arbitrary elements of $R$
for the other $x_j$, we obtain 
\begin{equation}\begin{split}\label{EqPnIdDecomp}P_n(\mathbb Z) \cap \Id(R,\mathbb Z)= 
  (\ch R) \mathbb Z x_1 x_2 \dots x_n\
\oplus \\ \bigoplus_{k=0}^{n-2} \bigoplus_{1\leqslant i_1 <
i_2 < \dots < i_k \leqslant n} \ x_{i_1} x_{i_2}\dots x_{i_k} \, \sigma_{i_1, \dots, i_k} 
\bigl(\Id(R,\mathbb Z) \cap \Gamma_{n-k}(\mathbb Z)\bigr).\end{split}\end{equation}

Combining~(\ref{EqPnDecomp}) and~(\ref{EqPnIdDecomp}), we get
$$\frac{P_n(\mathbb Z)}{P_n(\mathbb Z) \cap \Id(R, \mathbb Z)} \cong \bigoplus_{k=0}^n \bigoplus_{1\leqslant i_1 <
i_2 < \dots < i_k \leqslant n}  
\frac{\Gamma_{n-k}(\mathbb Z)}{\Gamma_{n-k}(\mathbb Z) \cap \Id(R, \mathbb Z)}
 $$
(direct sum of $\mathbb Z$-modules)
for an arbitrary ring $R$ with the unit $1_R$.
(We define $\frac{\Gamma_0(\mathbb Z)}{\Gamma_0(\mathbb Z) \cap \Id(R,\mathbb Z)} := \langle 1_R \rangle_{\mathbb Z} \subseteq R$.)
Calculating the number of the components, we obtain Theorem~\ref{TheoremCodimProperAndOrdinary}.
\end{proof}

\begin{corollary}
Let $R$ be a unitary ring.
 Then all multilinear polynomial identities
of $R$ are consequences of proper multilinear polynomial identities of $R$
and the identity $(\ch R) x \equiv 0$.
\end{corollary}
\begin{proof}
This follows from~(\ref{EqPnIdDecomp}).
\end{proof}

\begin{corollary}
Let $R$ be a unitary ring and let the sequence $\bigl(c_n(R, q)\bigr)_{n=1}^\infty$
be polynomially bounded for some $q$.
Then $c_n(R, q)$ is a polynomial in $n\in\mathbb N$.
\end{corollary}
\begin{proof}
If the sequence $\bigl(c_n(R, q)\bigr)_{n=1}^\infty$
is polynomially bounded, then by Theorem~\ref{TheoremCodimProperAndOrdinary} there exists $j_0 \in\mathbb N$
such that $\gamma_j(R, q)=0$ for all $j\geqslant j_0$. Now we apply Theorem~\ref{TheoremCodimProperAndOrdinary} once again.
\end{proof}

If $H$ is a subgroup of a group $G$
and $M$ is a left $\mathbb ZH$-module, then $M \uparrow G := \mathbb ZG \mathbin{\otimes_{\mathbb ZH}}M$.
The $G$-action on $\mathbb ZG \mathbin{\otimes_{\mathbb ZH}}
M$ is induced as follows: $g_0(g\otimes a) := g_0g \otimes a$
for $a\in M$, $g,g_0 \in G$.

Now we prove an analog of Drensky's theorem~\cite[Theorem 12.5.4]{DrenKurs}:
\begin{theorem}\label{TheoremZSnOrdinaryProper}
Let $R$ be a unitary ring, $\ch R = \ell$, $\ell \in\mathbb Z_+$. Consider for every $n\in \mathbb N$
the series of $\mathbb Z S_n$-submodules
$$M_0 :=\frac{P_n(\mathbb Z)}{P_n(\mathbb Z) \cap \Id(R, \mathbb Z)} 
\supsetneqq M_2 \supseteq M_3 \supseteq \dots \supseteq M_n \cong \frac{\Gamma_n(\mathbb Z)}{\Gamma_n(\mathbb Z)
 \cap \Id(R, \mathbb Z)}$$
 where each $M_k$ is the image of $\bigoplus_{t=k}^n\mathbb Z S_n (x_1\dots x_{n-t} \Gamma_t(\mathbb Z))$ and
 $M_{n+1}:=0$.
Then $M_0/M_2 \cong \mathbb Z_\ell$ (trivial $S_n$-action),
$$M_t/M_{t+1}\cong \left(\frac{\Gamma_t(\mathbb Z)}{\Gamma_t(\mathbb Z)
 \cap \Id(R, \mathbb Z)} \mathbin{\otimes_{\mathbb Z}} \mathbb Z\right) \uparrow S_n := \mathbb ZS_n \mathbin{\otimes_{\mathbb Z(S_t\times S_{n-t})}} \left(
 \frac{\Gamma_t(\mathbb Z)}{\Gamma_t(\mathbb Z) \cap \Id(R, \mathbb Z)} \otimes_{\mathbb Z} \mathbb Z\right)$$ for all $2\leqslant t \leqslant n$ where $S_{n-t}$ is permuting $x_{t+1}, \dots, x_n$
 and $\mathbb Z$ is a trivial $\mathbb ZS_{n-t}$-module.
\end{theorem}
\begin{proof}
First we notice that $M_0/M_2$ is generated by the image of $x_1 x_2 \dots x_n$.
Suppose the image of  $kx_1 x_2 \dots x_n$ belongs to $M_2$ for some $k\in\mathbb N$.
All the polynomials in $M_2$ vanish under the substitution $x_1=\dots = x_n = 1_R$
since each of them contain at least one commutator. Hence we get $k 1_R = 0$,
$\ell \mid k$, and $M_0/M_2 \cong \mathbb Z_\ell$.

Note that $\frac{\Gamma_t(\mathbb Z)}{\Gamma_t(\mathbb Z) \cap \Id(R, \mathbb Z)} \otimes_{\mathbb Z} \mathbb Z
\cong \frac{\Gamma_t(\mathbb Z)}{\Gamma_t(\mathbb Z) \cap \Id(R, \mathbb Z)}$ where $S_{n-t}$ acts trivially.
Consider the bilinear map $$\varphi \colon \mathbb ZS_n \times \frac{\Gamma_t(\mathbb Z)}{\Gamma_t(\mathbb Z) \cap \Id(R, \mathbb Z)} \to M_t/M_{t+1}$$
defined by $\varphi(\sigma, f)=x_{\sigma(t+1)}x_{\sigma(t+2)}
\dots x_{\sigma(n)}(\sigma f)$ for $\sigma \in S_n$, $f \in \frac{\Gamma_t(\mathbb Z)}{\Gamma_t(\mathbb Z) \cap \Id(R, \mathbb Z)}$.
Note that $\varphi(\sigma\pi, f)=\varphi(\sigma, \pi f)$ for all $\pi \in S_t\times S_{n-t}$
and $M_t/M_{t+1}$ is generated by all $\varphi(\sigma, f)$ for $\sigma \in S_n$
and $f\in \Gamma_t(\mathbb Z)$.

Suppose $L$ is an Abelian group and $\psi \colon \mathbb ZS_n \times \frac{\Gamma_t(\mathbb Z)}{\Gamma_t(\mathbb Z) \cap \Id(R, \mathbb Z)} \to L$ is a $\mathbb Z$-bilinear
map and $\psi(\sigma\pi, f)=\psi(\sigma, \pi f)$ for all $\pi \in S_t\times S_{n-t}$.
First we define $\bar \psi \colon M_t \to L$ on the elements that generate
 $M_t$ modulo $M_{t+1}$:
 $$\bar \psi(x_{i_1} x_{i_2}\dots x_{i_{n-t}}
f)=\psi(\sigma, \sigma^{-1} f)$$ where $\sigma \in S_n$ and $\sigma^{-1}f \in \frac{\Gamma_t(\mathbb Z)}{\Gamma_t(\mathbb Z) \cap \Id(R, \mathbb Z)}$ (e.g. we can take $\sigma(k) = i_k$ for $1 \leqslant k \leqslant n-t$). Clearly, $\bar\psi(x_{i_1} x_{i_2}\dots x_{i_{n-t}}f)$ does not depend on the choice
of $\sigma$.
 Suppose the image $\bar f_0$ of a polynomial $$f_0 = \sum_{i_1 < \dots < i_{n-t}} x_{i_1} x_{i_2}\dots x_{i_{n-t}} f_{i_1, \dots, i_{n-t}}$$ belongs to $M_{t+1}$ for some $f_{i_1, \dots, i_{n-t}}\in\Gamma_t(\mathbb Z)$. Substituting $$x_{i_1}=x_{i_2}=\dots = x_{i_{n-t}}=1_R$$
 and arbitrary values for the other $x_j$, we get zero for every $i_1 < \dots < i_{n-t}$.
Hence $f_{i_1, \dots, i_{n-t}} \in \Id(R, \mathbb Z)$ and $\bar\psi(\bar f_0)=0$.
Thus  we can define $\bar \psi$ to be zero on $M_{t+1}$ and  we may assume that $\bar \psi \colon M_t/M_{t+1} \to L$.

Note that $\bar\psi\varphi = \psi$. 
Hence $M_t/M_{t+1} \cong \mathbb ZS_n \mathbin{\otimes_{\mathbb Z(S_t\times S_{n-t})}} \left(
 \frac{\Gamma_t(\mathbb Z)}{\Gamma_t(\mathbb Z) \cap \Id(R, \mathbb Z)} \otimes_{\mathbb Z} \mathbb Z\right)$ (isomorphism of Abelian groups) where $\varphi(\sigma, f) \mapsto \sigma \otimes f$.
 Therefore, this is an isomorphism of $\mathbb ZS_n$-modules too.
\end{proof}

\section{A particular case of the Littlewood~--- Richardson rule}

Let $\mu \vDash n$, $\lambda \vdash n'$, $n'\leqslant n$. Suppose $\lambda_i \leqslant \mu_i$ for all $i\in\mathbb N$. Denote by $M(\mu)$ the free Abelian group generated by all $\mu$-tabloids.
Now we treat $D_\lambda$ as a Young subdiagram in $D_\mu$. Later on we always assume that in a pair
$(\lambda; \mu)$ we have $\lambda_1 = \mu_1$.

 Following~\cite[Definition 17.4]{JamesSymm},
we define a $\mathbb ZS_n$-submodule $S(\lambda, \mu)\subseteq M(\mu)$ where
$$S(\lambda; \mu) := \langle e^{\lambda, \mu}_{T_\mu}  \mid T_\mu \text{ is a tableau of the shape } \mu\rangle_\mathbb Z$$
and $e^{\lambda, \mu}_{T_\mu} := \sum_{\sigma \in C_{T_\lambda}} (\sign \sigma) \sigma [T_\mu]$.
Here $T_\lambda$ is the subtableau of $T_\mu$
defined by the partition $\lambda$ and $C_{T_\lambda} \subseteq S_n$ is the subgroup that leaves the numbers out of $T_\lambda$ invariant and puts every number from each column of $T_\lambda$ to the same column.
By $[T_\mu]$ we denote the tabloid corresponding to $T_\mu$. We assume $S(0;0)=0$
for the zero partitions $0\vdash 0$. Note that $S(\lambda; \lambda) \cong S(\lambda)$.
(The proof is completely analogous to the case when the coefficients are taken from a field.)

Let $F$ be a field and let $M^F(\mu)$ be the vector space over $F$ with the formal basis consisting
of all $\mu$-tabloids. In other words, $M^F(\mu)=M(\mu) \mathbin{\otimes_{\mathbb Z}}
F$. We define $S^F(\lambda; \mu)$
as the subspace in $M^F(\mu)$ generated by $S(\lambda; \mu) \mathbin \otimes 1$.

\begin{lemma}\label{LemmaNoTorsion} Let $\mu \vDash n$, $\lambda \vdash n'$, $n'\leqslant n$.  Suppose $\lambda_i \leqslant \mu_i$ for all $i\in\mathbb N$. Then $M(\mu)/S(\lambda; \mu)$ has no torsion.
\end{lemma}
\begin{proof} Recall that $M(\mu)$ is a finitely generated free Abelian group and 
$S(\lambda; \mu)$ is its subgroup. Hence we can choose a basis $a_1, a_2, \dots, a_t$
in $M(\mu)$ such that $m_1 a_1, m_2 a_2, \dots, m_k a_k$ is a basis of $S(\lambda; \mu)$
for some $m_i \in \mathbb N$. We claim that all $m_i=1$. First, we notice that
$a_1 \otimes 1, a_2 \otimes 1, \dots, a_t \otimes 1$ form a basis of 
$M^F(\mu)$ and $m_1 a_1 \otimes 1, m_2 a_2 \otimes 1, \dots, m_k a_k \otimes 1$
generate $S^F(\lambda; \mu)$ for any field $F$. Thus $\dim_F S^F(\lambda; \mu)
= k$ for $\ch F = 0$ and $\dim_F S^F(\lambda; \mu)
< k$ if $\ch F \mid m_i$ for at least one $m_i$. 
However, by~\cite[Theorem~17.13 (III)]{JamesSymm},  $\dim_F S^F(\lambda; \mu)$
does not depend on the field $F$. Therefore all $m_i=1$ and 
$M(\mu)/S(\lambda; \mu)$ is a free Abelian group.
\end{proof}
%
%

Let $c \geqslant 2$ be a natural number satisfying the following conditions:
$\mu_{c-1} = \lambda_{c-1}$ and $\mu_c > \lambda_c$.
Then we define the operators $A_c$ (``adding'') and $R_c$ (``raising'') in the following way:
\begin{enumerate}
\item if $\lambda_{c} =  \lambda_{c-1}$, then $A_c(\lambda; \mu)=(0;0)$ where $0 \vdash 0$ is a zero partition, otherwise $A_c(\lambda; \mu)=(\tilde\lambda; \mu)$
where $\tilde \lambda_i = \lambda_i$
for $i\ne c$ and $\tilde\lambda_{c} = \lambda_{c}+1$;
\item $R_c(\lambda; \mu)=(\tilde\lambda; \tilde\mu)$
where $\tilde \mu_i = \mu_i$
for $i\ne c-1,c$; $\tilde\mu_c = \lambda_c$, $\tilde\mu_{c-1} = \mu_{c-1}+(\mu_c-\lambda_c)$,
$\tilde\lambda_1 = \tilde\mu_1$ and
$\tilde \lambda_i = \lambda_i$
for $i > 1$.
\end{enumerate}

Fix $i \in\mathbb N$ and $0 \leqslant v \leqslant \mu_{i+1}$.
Let $\nu \vDash n$, $\nu_j = \mu_j$ for $j \ne i, i+1$, $\nu_i = \mu_i+\mu_{i+1}-v$,
$\nu_{i+1}=v$.
Then we define $\psi_{i,v} \in \Hom_{\mathbb ZS_n} (M(\mu),M(\nu))$
in the following way: $\psi_{i,v}[T_\mu] = \sum [T_\nu]$
where the summation runs over the set of all tabloids $[T_\nu]$ such that $[T_\nu]$
agrees with $[T_\mu]$ in all the rows except the $i$th and the $(i+1)$th,
and the $(i+1)$th is a subset of size $v$ of the $(i+1)$th row in $[T_\mu]$.
Analogously, we define $\psi^F_{i,v} \in \Hom_{FS_n} (M^F(\mu),M^F(\nu))$
for any field $F$.

\begin{lemma}\label{LemmaSlambdamuPsi}
\begin{enumerate}
\item $\psi_{c-1, \lambda_c} S(\lambda; \mu) =  S(R_c(\lambda; \mu))$;
\item $\ker \psi_{c-1, \lambda_c} \cap S(\lambda; \mu)=  S(A_c(\lambda; \mu))$.
\end{enumerate}
\end{lemma}
\begin{proof}
The proof of the first part of the lemma and of the embedding 
$\ker \psi_{c-1, \lambda_c} \supseteq  S(A_c(\lambda; \mu))$
is completely analogous to~\cite[Lemma 17.12]{JamesSymm}.
Now we notice that there exists a natural embedding $M(\lambda)\otimes 1 \subset M^{\mathbb Q}(\lambda)$.
By~\cite[Theorem 17.13]{JamesSymm}, $\ker \psi^{\mathbb Q}_{c-1, \lambda_c}  \cap S^{\mathbb Q}(\lambda; \mu) =  S^{\mathbb Q}(A_c(\lambda; \mu))$.
Thus if $\psi_{c-1, \lambda_c} a = 0$ for some $a\in S(\lambda; \mu)$,
then $ma \in  S(A_c(\lambda; \mu))$ for some $m \in\mathbb N$ and
 $a \in  S(A_c(\lambda; \mu))$
 since $M(\mu)/S(A_c(\lambda; \mu))$ is torsion-free by 
 Lemma~\ref{LemmaNoTorsion}.
\end{proof}

\begin{lemma}\label{LemmaSlambdamuSpechtSeries}
Let $n\in\mathbb N$, $\lambda \vdash n'$, $\mu \vDash n$, $n'\leqslant n$, $\lambda_i \leqslant \mu_i$
for all $i\in \mathbb N$. Then $S(\lambda; \mu)$
has a chain of submodules
$$S(\lambda; \mu) = M_0 \supsetneqq M_1 \supsetneqq M_2 \supsetneqq 
\dots \supsetneqq M_t = 0$$
 with factors $M_i/M_{i+1}$ isomorphic to Specht modules.
 Moreover, $S(\lambda; \mu)/M_i$ is torsion-free for any $i$.
\end{lemma}
\begin{proof}
If $\mu = \lambda$, then $S(\lambda; \mu) = S(\lambda)$ and there is nothing
to prove.  If $\mu \ne \lambda$, then we find $c\in\mathbb N$ such
that $\lambda_i = \mu_i$ for all $1 \leqslant i \leqslant c-1$
and $\lambda_c < \mu_c$. Since we always assume $\lambda_1 = \mu_1$,
we have $c \geqslant 2$. Now we apply Lemma~\ref{LemmaSlambdamuPsi}. Note that 
$\tilde\lambda_c > \lambda_c$ where $A_c(\lambda; \mu)=(\tilde\lambda; \tilde \mu)$ 
and $R_c$ moves the boxes of $D_\mu$ upper. 
Applying Lemma~\ref{LemmaSlambdamuPsi} many times, we get the first part of
Lemma~\ref{LemmaSlambdamuSpechtSeries} by induction.

Suppose $S(\lambda; \mu)/M_i$ is not torsion-free and $ma\in M_i$
for some $a \in S(\lambda; \mu)$, $a \notin M_i$, and $m\in\mathbb N$.
Then we can find an index $0 \leqslant k < i$ such that $a \in M_k$, $a \notin M_{k+1}$. However $ma\in M_i \subseteq M_{k+1}$.
i.e., the Specht module $M_k/M_{k+1}$ is not torsion-free either.
We get a contradiction since all Specht modules are subgroups in finitely generated free Abelian groups.
\end{proof} 

Now we can prove the $\mathbb Z$-analog of the particular case of the Littlewood~--- Richardson rule that sometimes is referred to as Young's rule~\cite[Theorem 2.3.3]{ZaiGia}, \cite[Theorem 12.5.2]{DrenKurs} and sometimes as Pieri's formula~\cite[(A.7)]{FultonHarris}. 

\begin{theorem}\label{TheoremYoungsRule}
Let $t, n\in\mathbb N$, $m\in\mathbb Z_+$, $t < n$, and $\lambda \vdash t$ and let $\mathbb Z$ be the trivial $\mathbb ZS_{n-t}$-module. Then $$\bigl(S(\lambda)/mS(\lambda)\bigr) \uparrow S_n := \mathbb ZS_n
\otimes_{\mathbb Z(S_{t} \times S_{n-t})} (\bigl(S(\lambda)/mS(\lambda)\bigr) \otimes_{\mathbb Z} \mathbb Z)$$ has a series of submodules with factors $S(\nu)/mS(\nu)$
where $\nu$ runs over the set of all partitions $\nu \vdash n$ such that 
$$\lambda_n \leqslant \nu_n \leqslant \lambda_{n-1} \leqslant \nu_{n-1}
\leqslant \dots \leqslant \lambda_2 \leqslant \nu_2 \leqslant \lambda_1 \leqslant \nu_1.$$
(Each factor occurs exactly once.)
\end{theorem}
\begin{proof} Suppose $\lambda=(\lambda_1, \dots, \lambda_s)$, $\lambda_s > 0$.
Then $S(\lambda) \uparrow S_n \cong S(\lambda; \mu)$
where $\mu=(\lambda_1, \dots, \lambda_s, n-t)$.
Now Lemma~\ref{LemmaSlambdamuSpechtSeries} implies the theorem for $m=0$.

Suppose $m > 0$. Then $\bigl(S(\lambda)/mS(\lambda)\bigr) \uparrow S_n \cong
\bigl(S(\lambda) \uparrow S_n\bigr) / \bigl(m(S(\lambda) \uparrow S_n)\bigr)$.
Let $$S(\lambda) \uparrow S_n = M_0 \supsetneqq M_1 \supsetneqq M_2 \supsetneqq 
\dots \supsetneqq M_t = 0$$
where $M_{i-1}/M_i \cong S(\lambda^{(i)})$, $\lambda^{(i)} \vdash n$, $1\leqslant i
\leqslant t$.

Hence $$\bigl(S(\lambda) \uparrow S_n\bigr) / \bigl(m(S(\lambda) \uparrow S_n)\bigr) = \overline{M_0} \supsetneqq \overline{M_1} \supsetneqq \overline {M_2} \supsetneqq 
\dots \supsetneqq \overline{M_t} = 0$$
where $\overline{M_i} \cong (M_i + m(S(\lambda) \uparrow S_n))/m(S(\lambda) \uparrow S_n)$
and \begin{equation*}\begin{split}
\overline{M_{i-1}} / \overline{M_i} \cong (M_{i-1} + m(S(\lambda) \uparrow S_n))/(M_i + m(S(\lambda) \uparrow S_n))  \cong \\   M_{i-1}/M_{i-1}\cap(M_i + m(S(\lambda) \uparrow S_n)) 
 = \\ M_{i-1}/(M_i + M_{i-1}\cap m(S(\lambda) \uparrow S_n))
 \cong (M_{i-1}/M_i)/((M_i + M_{i-1}\cap m(S(\lambda) \uparrow S_n))/M_i).
\end{split}\end{equation*}
 By Lemma~\ref{LemmaSlambdamuSpechtSeries},
 $(S(\lambda) \uparrow S_n) / M_{i-1}$ is torsion-free.
 Hence
$M_{i-1}\cap m(S(\lambda) \uparrow S_n) =  mM_{i-1}$ and \begin{equation*}\begin{split}\overline{M_{i-1}} / \overline{M_i} \cong 
(M_{i-1}/M_i)/((M_i + mM_{i-1})/M_i) =\\ (M_{i-1}/M_i)/(m(M_{i-1}/M_i))\cong S(\lambda^{(i)})/mS(\lambda^{(i)}).\end{split}\end{equation*}
The description of $\lambda^{(i)}$ is obtained from the proof of Lemma~\ref{LemmaSlambdamuSpechtSeries}.
\end{proof}

\section{Algebras of upper triangular matrices}\label{SectionUT2R}

\subsection{Codimensions and multilinear identities}\label{SubsectionCodimUT2R}

Let $M$ be an $(R_1, R_2)$-bimodule for commutative rings $R_1$, $R_2$ with $1$
and let $R = \left(\begin{array}{rr} R_1 & M \\ 0 & R_2\end{array}\right)$.

In this section, we calculate $c_n(R, q)$ for all $q = p^k$ and $q=0$,
describe the structure of the $\mathbb Z S_n$-module
$\frac{P_n(\mathbb Z)}{P_n(\mathbb Z) \cap \Id(R, \mathbb Z)}$
and find such multilinear polynomials that elements of $\Id(R, \mathbb Z) \cap P_n(\mathbb Z)$
are consequences of them.
\begin{remark}
If $F$ is a field of characteristic $0$ and $A=\UT_2(F):=\left(\begin{array}{rr} F & F \\ 0 & F\end{array}\right)$,
then $c_n(A, F)$ and generators of $\Id(A, F)$ as a $T$-ideal can be found, e.g., in~\cite[Theorem 4.1.5]{ZaiGia}. The structure of the $F S_n$-module
$\frac{P_n(F)}{P_n(F) \cap \Id(A, F)}$ can be determined using proper cocharacters~\cite[Theorem 12.5.4]{DrenKurs}.
\end{remark}

\begin{theorem}\label{TheoremCodimUT2R}
All polynomials from $P_n(\mathbb Z) \cap \Id(R, \mathbb Z)$, $n\in\mathbb N$,
are consequences of the left hand sides of the following polynomial
identities in $R$:
\begin{equation}\label{EqId1}[x,y][z,t]\equiv 0,
\end{equation} \begin{equation}\label{EqId2}\ell x\equiv 0,\end{equation}
 \begin{equation}\label{EqId3}m [x,y]=0
 \end{equation}
where $[x,y]:= xy-yx$, $$\ell := \min\left\lbrace n \in\mathbb N \mid na = 0 \text{ for all } a\in R_1\cup R_2 \right\rbrace,$$ $$m := \min\left\lbrace n \in\mathbb N \mid na = 0 \text{ for all } a\in M \right\rbrace.$$ (If one of the corresponding sets is empty, we define $\ell=0$ or $m=0$, respectively. Note that $m \mid \ell$.)

 Moreover, $\frac{P_n(\mathbb Z)}{P_n(\mathbb Z) \cap \Id(R, \mathbb Z)} \cong \mathbb Z_\ell \oplus \mathbb (\mathbb Z_m)^{(n-2)2^{n-1}+1}$ where $\mathbb Z_0 := \mathbb Z$.
\end{theorem}
\begin{remark}
Now $c_n(R, q)$ can be easily computed.
If $R_1=R_2=M$ and $R_1 = R_2$ is a field, we obtain the same numbers as in~\cite[Theorem 4.1.5]{ZaiGia}.
\end{remark}
\begin{proof}[Proof of Theorem~\ref{TheoremCodimUT2R}.]
Denote by $e_{ij}$ the matrix units. Then $R= R_1 e_{11} \oplus R_2 e_{22}\oplus M e_{12}$
(direct sum of subspaces), $[R,R]\subseteq M e_{12}$, and (\ref{EqId1})--(\ref{EqId3}) are indeed
polynomial identities of $R$.

Now we consider an arbitrary mononomial from $P_n(\mathbb Z)$ and find the first inversion among the indexes of its variables.
We replace the corresponding pair of variables with the sum of their commutator and their product in the right order.
Note that $[x,y]u[z,t]=[x,y][z,t]u+[x,y][u,[z,t]] \equiv 0$ is a consequence of~(\ref{EqId1}).
Therefore, we may assume that all the variables to the right of the commutator have increasing indexes. For example: $$ \begin{array}{lcl} 
x_{3}x_{1}x_{4} x_{2} & = & x_{1} x_{3} x_{4} x_{2} + [x_{3}, x_{1}] x_{4} x_{2} \\
& \stackrel{(\ref{EqId1})}{\equiv} & x_{1} x_{3} x_{2} x_{4} +  x_{1} x_{3}[x_{4}, x_{2}] +  [x_{3}, x_{1}] x_{2} x_{4}  \\
& = & x_{1} x_{2} x_{3} x_{4} + x_{1} [x_{3}, x_{2}] x_{4} + x_{1} x_{3}[x_{4}, x_{2}] +  [x_{3}, x_{1}] x_{2} x_{4}.  \\
\end{array}$$
Continuing this procedure,
we present any element of $P_n(\mathbb Z)$
modulo the consequences of~(\ref{EqId1}) as a linear combination
of polynomials $f_0:=x_1 x_2 \dots x_n$ and  \begin{equation}\label{EqPolynomial}
x_{i_{1}}\dots x_{i_{k}}[x_{s},x_{r}]x_{j_{1}}\dots x_{j_{n-k-2}} \text{ for } i_{1} < \dots < i_{k}<s,\ r < s,\ j_{1} < \dots < j_{n-k-2}.\end{equation}
Denote the set of polynomials~(\ref{EqPolynomial}) by $\Xi$.

Consider the free Abelian group $\mathbb Z(\Xi \cup \lbrace f_0\rbrace)$
with the basis $\Xi \cup \lbrace f_0\rbrace$.
Now we have the surjective homomorphism $\varphi \colon \mathbb Z(\Xi \cup \lbrace f_0\rbrace)
\to \frac{P_n(\mathbb Z)}{P_n(\mathbb Z) \cap \Id(R, \mathbb Z)}$
where $\varphi(f)$ is the image of $f \in \Xi \cup \lbrace f_0\rbrace$ in 
$\frac{P_n(\mathbb Z)}{P_n(\mathbb Z) \cap \Id(R, \mathbb Z)}$.
We claim that $\ker \varphi$ is generated by $\ell f_0$ and all $m f$ where $f\in \Xi$.

Suppose that a linear combination $f_1$ of $f_0$ and elements from $\Xi$ is a polynomial identity,
however $f_1$ is not a linear combination of $\ell f_0$ and $m f$, $f\in \Xi$.
If we substitute $$x_1 =x_2 = \dots = x_n = 1_{R_i}e_{ii} \text{ where } i \in \lbrace 1,2 \rbrace,$$
all $f \in \Xi$ vanish. Therefore, the coefficient of $f_0$ is a multiple of $\ell$.
Now we find $f_2 := x_{i_{1}}\dots x_{i_{k}}[x_{s},x_{r}]x_{j_{1}}\dots x_{j_{n-k-2}} \in \Xi$
with the greatest $k$ such that the coefficient $\beta$ of $f_2$ in $f_1$ is not a multiple of $m$.
Then we substitute $x_{i_{1}}=\dots =x_{i_{k}}=x_s = 1_{R_1}e_{11}$, $x_r = a e_{12}$, $x_{j_{1}}=\dots = x_{j_{n-k-2}}=1_{R_1}e_{11}+1_{R_2}e_{22}=1_R$ where $a\in M$ and $\beta a \ne 0$.
Our choice of $f_2$ implies that $f_2$ is the only summand in $f_1$ that could be nonzero under this substitution. Hence $f_1$ does not vanish and we get a contradiction.
Therefore, $\ker \varphi$ is generated by $\ell f_0$ and $m f$, $f\in \Xi$.
In particular,
$\frac{P_n(\mathbb Z)}{P_n(\mathbb Z) \cap \Id(R, \mathbb Z)} \cong \mathbb Z_\ell \oplus  (\mathbb Z_m)^{|\Xi|}$ and every multilinear polynomial identity of $R$ is a consequence
of (\ref{EqId1})--(\ref{EqId3}).

Note that \begin{equation*}\begin{split}|\Xi|=\sum_{k=2}^n (k-1)\binom{n}{k} = \sum_{k=2}^n \frac{n!}{(k-1)!(n-k)!} -
\sum_{k=2}^n \binom{n}{k}=\\ n\sum_{k=1}^{n-1} \frac{(n-1)!}{k!(n-k-1)!} - (2^n-n-1)=n(2^{n-1}-1)- (2^n-n-1)
= (n-2)2^{n-1}+1\end{split}\end{equation*}
and the theorem follows.
\end{proof}


\begin{corollary}
Multilinear polynomial identities of $\UT_2(\mathbb Q)$ as a ring are generated by~(\ref{EqId1}).
\end{corollary}

\subsection{$\mathbb ZS_n$-modules}

Note that the Jacobi identity and~(\ref{EqId1}) imply that $\frac{\Gamma_n(\mathbb Z)}{\Gamma_n(\mathbb Z)\cap\Id(R,\mathbb Z)}$ is generated as a $\mathbb Z$-module by $[x_i, x_n, x_1, x_2, \dots, \hat x_i, \dots, x_{n-1}]$ where $1 \leqslant i \leqslant n-1$.

\begin{lemma}\label{LemmaSymmKryuk}
Let $R$ be the ring from Subsection~\ref{SubsectionCodimUT2R}
and $T_\lambda = \begin{array}{|l|l|l|l|} \cline{1-4} 1 & 2 & \dots & n-1
\\ \cline{1-4} n \\ \cline{1-1} \end{array}$.
Then 
\begin{equation}\label{EqSymm}b_{T_\lambda} a_{T_\lambda} [x_1, x_n, x_2, x_3, \dots, x_{n-1}]
\equiv n(n-2)![x_1, x_n, x_2, x_3, \dots, x_{n-1}]\ (\mathop\mathrm{mod} P_n(\mathbb Z)\cap \Id(R)).\end{equation}
\end{lemma}
\begin{proof} Indeed,
\begin{equation*}\begin{split}b_{T_\lambda} a_{T_\lambda} [x_1, x_n, x_2, x_3, \dots, x_{n-1}]
\equiv b_{T_\lambda} (n-2)! \sum_{i=1}^{n-1} [x_i, x_n, x_1, x_2, \dots, \hat x_i, \dots, x_{n-1}]
= \\ (n-2)! \sum_{i=2}^{n-1} \left([x_i, x_n, x_1, x_2, \dots, \hat x_i, \dots, x_{n-1}]
-  [x_i, x_1, x_n, x_2, \dots, \hat x_i, \dots, x_{n-1}]\right) 
+ \\2(n-2)! [x_1, x_n, x_2, x_3, \dots, x_{n-1}]\equiv n(n-2)![x_1, x_n, x_2, x_3, \dots, x_{n-1}]\end{split}\end{equation*}
since, by the Jacobi identity, $[x_i, x_1, x_n]=[x_i, x_n, x_1]+[x_n, x_1, x_i]$.
\end{proof}

First, we determine the structure of $\frac{\Gamma_n(\mathbb Z)}{\Gamma_n(\mathbb Z)\cap\Id(R,\mathbb Z)}$
for $R=\UT_2(\mathbb Q)$.

\begin{lemma}\label{LemmaIsoZSnUT2Q} Let $T_\lambda = \begin{array}{|l|l|l|l|} \cline{1-4} 1 & 2 & \dots & n-1
\\ \cline{1-4} n \\ \cline{1-1} \end{array}$. Then
$\frac{\Gamma_n(\mathbb Z)}{\Gamma_n(\mathbb Z)\cap\Id(\UT_2(\mathbb Q),\mathbb Z)}
\cong (\mathbb ZS_n)b_{T_\lambda}a_{T_\lambda}$.
\end{lemma}
\begin{proof}
We claim that if $u b_{T_\lambda}a_{T_\lambda} = 0$ for some $u\in\mathbb ZS_n$,
then $$u [x_1, x_n, x_2, x_3, \dots, x_{n-1}] \in \Gamma_n(\mathbb Z) \cap \Id(\UT_2(\mathbb Q),\mathbb Z).$$
Indeed, by~(\ref{EqSymm}), $$n(n-2)!\, u [x_1, x_n, x_2, x_3, \dots, x_{n-1}] \equiv
u b_{T_\lambda} a_{T_\lambda} [x_1, x_n, x_2, x_3, \dots, x_{n-1}] = 0.$$
Since $\UT_2(\mathbb Q)$ has no torsion, $u [x_1, x_n, x_2, x_3, \dots, x_{n-1}] \equiv 0$
is a polynomial identity of $\UT_2(\mathbb Q)$.

Thus we can define the surjective homomorphism $\varphi \colon (\mathbb ZS_n)b_{T_\lambda}a_{T_\lambda}
\to \frac{\Gamma_n(\mathbb Z)}{\Gamma_n(\mathbb Z)\cap\Id(\UT_2(\mathbb Q),\mathbb Z)}$
by $\varphi(\sigma b_{T_\lambda}a_{T_\lambda}) = \sigma [x_1, x_n, x_2, x_3, \dots, x_{n-1}]$
for $\sigma \in S_n$. 

Analogously, we can define the surjective homomorphism
$$\varphi_0 \colon (\mathbb QS_n)b_{T_\lambda}a_{T_\lambda}
\to \frac{\Gamma_n(\mathbb Q)}{\Gamma_n(\mathbb Q)\cap\Id(\UT_2(\mathbb Q),\mathbb Q)}$$
by $\varphi(\sigma b_{T_\lambda}a_{T_\lambda}) = \sigma [x_1, x_n, x_2, x_3, \dots, x_{n-1}]$
for $\sigma \in S_n$. Since $(\mathbb QS_n)b_{T_\lambda}a_{T_\lambda}$ is an irreducible
$\mathbb QS_n$-module, $\varphi_0$ is an isomorphism of $\mathbb QS_n$-modules.
We claim that $\varphi$ is an isomorphism of $\mathbb ZS_n$-modules.

Indeed, suppose $$u [x_1, x_n, x_2, x_3, \dots, x_{n-1}] \in \Gamma_n(\mathbb Z) \cap \Id(\UT_2(\mathbb Q), \mathbb Z)$$ for some $u\in\mathbb ZS_n$. Then $\varphi_0(u b_{T_\lambda}a_{T_\lambda})=u [x_1, x_n, x_2, x_3, \dots, x_{n-1}] \in \Gamma_n(\mathbb Q) \cap \Id(\UT_2(\mathbb Q),\mathbb Z)$ and $u b_{T_\lambda}a_{T_\lambda}=0$.
Hence $\varphi$ is an isomorphism and the lemma is proven.
\end{proof}

\begin{theorem}\label{TheoremProperZSnUT2}
Let $R$ and $m$ be, respectively, the ring and the number from Subsection~\ref{SubsectionCodimUT2R}.
Then $\frac{\Gamma_n(\mathbb Z)}{\Gamma_n(\mathbb Z)\cap\Id(R,\mathbb Z)} \cong S(\lambda)/mS(\lambda)$
where $\lambda = (n-1, 1)$, for all $n \geqslant 2$.
\end{theorem}
\begin{proof} 
Recall that $\frac{\Gamma_n(\mathbb Z)}{\Gamma_n(\mathbb Z)\cap\Id(R,\mathbb Z)}$ is generated as a $\mathbb Z$-module by $[x_i, x_n, x_1, x_2, \dots, \hat x_i, \dots, x_{n-1}]$ where $1 \leqslant i \leqslant n-1$.
We exploit the same trick as in the proof of Theorem~\ref{TheoremCodimUT2R}.
Using the substitution $x_1=\dots =x_{i-1}=x_{i+1}= \dots = x_n = 1_{R_1}e_{11}$, $x_i = a e_{12}$ where $a\in M$, we obtain that $\frac{\Gamma_n(\mathbb Z)}{\Gamma_n(\mathbb Z)\cap\Id(R,\mathbb Z)}$
is the direct sum of $n-1$ cyclic groups isomorphic to $\mathbb Z_m$ and generated by 
$[x_i, x_n, x_1, x_2, \dots, \hat x_i, \dots, x_{n-1}]$ where $1 \leqslant i \leqslant n-1$.

By Theorem~\ref{TheoremCodimUT2R} and its corollary, we have the natural surjective
homomorphism $\frac{\Gamma_n(\mathbb Z)}{\Gamma_n(\mathbb Z)\cap\Id(\UT_2(\mathbb Q),\mathbb Z)}
\to \frac{\Gamma_n(\mathbb Z)}{\Gamma_n(\mathbb Z)\cap\Id(R,\mathbb Z)}$.
The remarks above imply that the kernel equals $m\frac{\Gamma_n(\mathbb Z)}{\Gamma_n(\mathbb Z)\cap\Id(\UT_2(\mathbb Q),\mathbb Z)}$. Now the theorem follows from Lemma~\ref{LemmaIsoZSnUT2Q}.
\end{proof}

Applying Theorems~\ref{TheoremZSnOrdinaryProper}, \ref{TheoremYoungsRule}, and \ref{TheoremProperZSnUT2}
we immediately get

\begin{theorem}\label{TheoremOrdinaryZSnUT2}
Let $R$, $\ell$, and $m$ be, respectively, the ring and the numbers from Subsection~\ref{SubsectionCodimUT2R}.
Then there exists a chain of $\mathbb ZS_n$-submodules in $\frac{P_n(\mathbb Z)}{P_n(\mathbb Z)\cap\Id(R,\mathbb Z)}$  with the set of factors that
consists of one copy of $\mathbb Z_\ell$ and $(\lambda_1-\lambda_2+1)$ copies of $S(\lambda_1, \lambda_2, \lambda_3)/mS(\lambda_1, \lambda_2, \lambda_3)$
where $(\lambda_1, \lambda_2, \lambda_3) \vdash n$, $\lambda_2 \geqslant 1$, $\lambda_3 \in \lbrace 0,1 \rbrace$.
\end{theorem}

\section{Grassmann algebras}\label{SectionGrassmannR}

Let $R$ be a commutative ring with a unit element $1_R$, $\ch R = \ell$ where either $\ell$ is an odd natural number or $\ell = 0$.
  We define the Grassman algebra~$G_R$ over a ring $R$ as the $R$-algebra with a unit, generated by
  the countable set of generators $e_i$, $i\in\mathbb N$, and the anti-commutative relations
  $e_{i} e_{j} = -e_{j}e_{i}$, $i,j\in\mathbb N$.
 Here we consider the same questions as for the upper triangular matrices.
 \subsection{Codimensions and polynomial identities} \label{grassmann}
This lemma is known but we provide its proof for the reader's convenience.
\begin{lemma}\label{identities}
The polynomial identity $[y,x][z,t] + [y,z][x,t]\equiv 0$
is a consequence of $[x_{1},x_{2},x_{3}] \equiv 0$.
In particular, $[x,y]u[z,t] + [x,t]u[z,y]\equiv 0$
for all $u \in \mathbb Z \langle X \rangle$.
\end{lemma}
\begin{proof}
Note that \begin{equation}\begin{split}
 [x, yt, z] = [[x, y]t, z]+[y[x, t], z]
=  [x, y, z]t+ [x, y][t, z]+ \\ [y,z][x, t]+y[x,t,z]\equiv [x, y][t, z]+[y,z][x, t]
= [y, x][z, t]+[y,z][x, t]\end{split}
\end{equation}
modulo $[x_{1},x_{2},x_{3}] \equiv 0$. (Here we have used Jacobi's identity too.)
Hence
\begin{equation}\begin{split}[x,y]u[z,t] + [x,t]u[z,y] = [x,y][u,[z,t]] + [x,y][z,t]u +\\ [x,t][u,[z,y]]
+[x,t][z,y]u\equiv [x,y][z,t]u + [x,t][z,y]u \equiv 0.\end{split}
\end{equation}

\end{proof}
\begin{theorem}\label{TheoremCodimGrass}
All polynomials from $P_n(\mathbb Z) \cap \Id(G_{R}, \mathbb Z)$, $n\in\mathbb N$,
are consequences of the left hand sides of the following polynomial
identities in $R$:
\begin{equation}\label{EqId1bis}[x,y,z]\equiv 0,
\end{equation} \begin{equation}\label{EqId2bis}\ell x\equiv 0.\end{equation}
 Moreover, $\frac{P_n(\mathbb Z)}{P_n(\mathbb Z) \cap \Id(G_{R}, \mathbb Z)} \cong \mathbb (\mathbb Z_\ell)^{2^{n-1}}$.
\end{theorem}
\begin{proof}
Define $G_R^{(0)} = \langle e_{i_1} e_{i_2}\dots e_{i_{2k}} \mid k\in\mathbb Z_+\rangle_R$
and $G_R^{(1)} = \langle e_{i_1} e_{i_2}\dots e_{i_{2k+1}} \mid k\in\mathbb Z_+\rangle_R$.
Clearly, $G_R= G_R^{(0)} \oplus G_R^{(1)}$ (direct sum of $R$-submodules), $[G_R,G_R]\subseteq G_R^{(0)}$,
  $G_R^{(0)} = Z (G_{R})$.
   Hence $[x_{1},x_{2},x_{3}]\equiv 0$ is a polynomial identity. 
   Obviously,  (\ref{EqId2bis}) is a polynomial identity too.
   
Let  \begin{equation*}\begin{split}\Xi = \{ x_{i_{1}} \dots x_{i_{k}} [x_ {j_{1}},x_{j_{2}}] \dots [x_{j_{2m-1}}, x_{j_{2m}}] \mid i_{1} < \dots < i_{k}, \\ j_{1} < \dots < j_{2m},\ k+2m = n,\ k,m \in \mathbb Z_+ \}\subset P_n(\mathbb Z).\end{split}\end{equation*}
   
By Lemma~\ref{identities}, every
polynomial from $P_n(\mathbb Z)$ can be presented modulo (\ref{EqId1bis})
as a linear combination of polynomials from $\Xi$.
For example,  \begin{equation*}\begin{split}x_3 x_2 x_4 x_1 =  -[x_2, x_3] x_4 x_1
+ x_2 x_3 x_4 x_1 = \\ ([x_2, x_3] [x_1, x_4] - [x_2, x_3] x_1 x_4)
+ (x_2 x_3 x_1 x_4 - x_2 x_3 [x_1, x_4]) \equiv  \\
-[x_2, x_1] [x_3, x_4] - x_1 x_4 [x_2, x_3]
+ x_2 x_1 x_3 x_4 - x_2 [x_1, x_3] x_4 - x_2 x_3 [x_1, x_4]
\equiv \\  [x_1, x_2] [x_3, x_4] - x_1 x_4 [x_2, x_3] +
 x_1 x_2 x_3 x_4 - [x_1, x_2] x_3 x_4 - x_2  x_4 [x_1, x_3] - x_2 x_3 [x_1, x_4]
\equiv  \\ [x_1, x_2] [x_3, x_4] - x_1 x_4 [x_2, x_3]+
 x_1 x_2 x_3 x_4 - x_3 x_4[x_1, x_2]  - x_2  x_4 [x_1, x_3] - x_2 x_3 [x_1, x_4]
.\end{split}\end{equation*}

Consider the free Abelian group $\mathbb Z\Xi$
with the basis $\Xi$.
Now we have the surjective homomorphism $\varphi \colon \mathbb Z\Xi
\to \frac{P_n(\mathbb Z)}{P_n(\mathbb Z) \cap \Id(G_R, \mathbb Z)}$
where $\varphi(f)$ is the image of $f \in \Xi $ in 
$\frac{P_n(\mathbb Z)}{P_n(\mathbb Z) \cap \Id(G_R, \mathbb Z)}$.
We claim that $\ker \varphi$ is generated by $\ell f$ where $f\in \Xi$.

Suppose that a linear combination $f_1$ of elements from $\Xi$ is a polynomial identity,
however $f_1$ is not a linear combination of $\ell f$, $f\in \Xi$.
Now we find $$f_2 := x_{i_{1}} \dots x_{i_{k}} [x_ {j_{1}},x_{j_{2}}] \dots [x_{j_{2m-1}}, x_{j_{2m}}] \in \Xi$$
with the greatest $k$ such that the coefficient $\beta$ of $f_2$ in $f_1$ is not a multiple of $\ell$.
Then we substitute $x_{i_{1}}=\dots =x_{i_{k}}= 1_{G_R}$,  $x_{j_i}=e_i$, $1\leqslant i \leqslant 2m$.
Our choice of $f_2$ implies that $f_2$ is the only summand  in $f_1$ that could be nonzero under this substitution. Hence the value of $f_1$ equals $ (2^m \beta\,  1_R) e_1 e_2 \dots e_m = 0$.
However, $G_R$ is a free $R$-module and $e_1 e_2 \dots e_m$ is one of its basis elements.
Therefore $2^m \beta\,  1_R = 0$, $\ell \mid (2^m \beta)$ and $\ell \mid \beta$ since $2 \nmid \ell$.
We get a contradiction.

Thus $\ker \varphi$ is generated by $\ell f$, $f\in \Xi$.
In particular,
$\frac{P_n(\mathbb Z)}{P_n(\mathbb Z) \cap \Id(G_R, \mathbb Z)} \cong  (\mathbb Z_\ell)^{|\Xi|}$ and every multilinear polynomial identity of $G_R$ is a consequence
of (\ref{EqId1bis}) and (\ref{EqId2bis}).

We now calculate $| \Xi |$.
The number of these polynomials equals the number of choices of $x_{i_{1}}, \dots, x_{i_{k}}$. If $n$ is odd, this number equals
${n \choose 1} + {n \choose 3} + \dots + {n \choose n}.$
If $n$ is even, the number equals ${n \choose 0} + {n \choose 2} + \dots + {n \choose n}$. But the both are equal to $2^{n-1}$. Indeed, denote $s_{0} = \sum\limits_{i \mbox{ \tiny{even}}} { n \choose i }$ and $s_{1} = \sum\limits_{i \mbox{ \tiny{odd}}} { n \choose i }$. Then $2^{n} = (1+1)^{n} = s_{0} + s_{1}$ and $0 =(1-1)^{n}=s_{0} -s_{1}$. So $| \Xi |=s_{0}=s_{1}= 2^{n-1}$.
\end{proof}

\subsection{$\mathbb{Z}S_{n}$-modules}

First we determine the structure of $\mathbb{Z}S_{n}$-modules of proper polynomial
functions.
\begin{theorem} \label{TheoremProperZSnR}
 Let $G_{R}$ be the Grassmann algebra over $R$. 
Let  $\lambda = (1^{2m})$ and $T_\lambda = \begin{array}{|l|} \cline{1-1} 1 \\  \cline{1-1}  2 \\  \cline{1-1}  \vdots \\  \cline{1-1} 2m\\ \cline{1-1} \end{array}$. Then $\frac{\Gamma_{2m}(\mathbb Z)}{\Gamma_{2m}(\mathbb Z) \cap \Id(G_{R}, \mathbb Z)} \cong  S(\lambda)/  \ell S(\lambda) $ for all $m\in\mathbb N$, where $\ell = \ch R$,
and $\frac{\Gamma_{2m+1}(\mathbb Z)}{ \Gamma_{2m+1}(\mathbb Z) \cap \Id(G_R, \mathbb Z)}=0$
for all $m\in\mathbb Z_+$.
\end{theorem}
\begin{proof}
Note $S(\lambda)$ is a free cyclic group generated by $b_{T_\lambda}[T_\lambda]$
and $\sigma b_{T_\lambda}[T_\lambda] = (\sign \sigma) b_{T_\lambda}[T_\lambda]$
for all $\sigma \in S_n$.

The proof of Theorem~\ref{TheoremCodimGrass} implies that $\frac{\Gamma_{2m}(\mathbb Z)}{\Gamma_{2m}(\mathbb Z) \cap \Id(G_{R}, \mathbb Z)} \cong \mathbb Z_{\ell}$ is a cyclic group generated by $[x_ {1},x_{2}] \dots [x_{2m-1},x_{2m}]$. 
By Lemma~\ref{identities}, $$\sigma [x_{1}, x_{2}] \dots [x_{2m-1},x_{2m}] = (\sign \sigma) [x_{1}, x_{2}] \dots [x_{2m-1},x_{2m}] \text{ for all }\sigma \in S_n.$$
Hence $\frac{\Gamma_{2m}(\mathbb Z)}{\Gamma_{2m}(\mathbb Z) \cap \Id(G_R, \mathbb Z)} \cong S(\lambda)/ \ell S(\lambda)$.
The first assertion is proved.
The second assertion is evident since every long commutator of length greater than $2$
is a polynomial identity of $G_R$.
\end{proof}

\begin{theorem}\label{TheoremOrdinaryZSnGrass}
Let $G_{R}$ be the Grassmann algebra over the $R$. Then there exists a chain of $\mathbb Z S_n$-submodules in $\frac{P_n(\mathbb Z)}{P_n(\mathbb Z)\cap\Id(G_{R},\mathbb Z)}$  with factors  $S(n-k, 1^{k})/\ell S(n-k, 1^{k})$ for each $0 \leqslant k \leqslant n-1$ (each factor occurs exactly once)
where $\ell = \ch R$.\end{theorem}
\begin{proof}
Now we apply Theorems~\ref{TheoremZSnOrdinaryProper}, \ref{TheoremYoungsRule}, and \ref{TheoremProperZSnR}.
By Theorem~\ref{TheoremYoungsRule}, a diagram consisting of a single column can generate
only diagrams $D_{(n-k, 1^{k})}$. Since we have diagrams of an even length only, each factor occurs only once.
\end{proof}

\section*{Acknowledgements}

The authors are grateful to Eric Jespers and Mikhail Zaicev for helpful discussions.
They also appreciate the referee for his remarks.

\end{document}